\documentclass[12pt,reqno]{article}

\usepackage[usenames]{color}
\usepackage[colorlinks=true,
linkcolor=webgreen, filecolor=webbrown,
citecolor=webgreen]{hyperref}

\definecolor{webgreen}{rgb}{0,.5,0}
\definecolor{webbrown}{rgb}{.6,0,0}

\usepackage{amssymb}
\usepackage{extarrows}
\usepackage{graphicx}
\usepackage{amscd}
\usepackage{lscape}
\usepackage{tikz}
\usepackage{tikz-cd}
\usepackage{pgfplots}

\usetikzlibrary{matrix}
\usetikzlibrary{fit,shapes}
\usetikzlibrary{positioning, calc}
\tikzset{circle node/.style = {circle,inner sep=1pt,draw, fill=white},
        X node/.style = {fill=white, inner sep=1pt},
        dot node/.style = {circle, draw, inner sep=5pt}
        }
\usepackage{tkz-fct}

\usepackage{amsthm}
\newtheorem{theorem}{Theorem}
\newtheorem{lemma}[theorem]{Lemma}
\newtheorem{proposition}[theorem]{Proposition}
\newtheorem{corollary}[theorem]{Corollary}
\newtheorem{conjecture}[theorem]{Conjecture}

\theoremstyle{definition}

\newtheorem{example}[theorem]{Example}

\usepackage{float}

\usepackage{graphics,amsmath}
\usepackage{amsfonts}
\usepackage{latexsym}
\usepackage{epsf}

\newcommand{\seqnum}[1]{\href{http://oeis.org/#1}{\underline{#1}}}

\begin{document}

\begin{center}
\vskip 1cm{\LARGE\bf On a Central Transform of Integer Sequences} \vskip 1cm \large
Paul Barry\\
School of Science\\
Waterford Institute of Technology\\
Ireland\\
\href{mailto:pbarry@wit.ie}{\tt pbarry@wit.ie}\\
\end{center}
\vskip .2 in

\begin{abstract} \noindent We use the concept of the half of a lower-triangular matrix to define a transformation on integer sequences. We explore the properties of this transformation, including in some cases a study of the Hankel transform of the transformed sequences. Starting from simple sequences with elementary rational generating functions, we obtain many sequences of combinatorial significance. We make extensive use of techniques drawn from the theory of Riordan arrays. \end{abstract}

\section{Introduction}
In this note, we work within the context of Riordan arrays \cite{book, SGWW} and their halves \cite{halves}. As we wish to obtain sequences of combinatorial significance, we work over the ring of integers $\mathbb{Z}$, and so we work with the matrix representation of the (integer) Riordan group defined as follows. Elements of the group are defined by a pair $(g,f)$ of power series in $\mathbf{Z}[[x]]$, where
$$g(x)=1+g_1x+g_2x^2+g_3x^3+\cdots,$$ and
$$f(x)=x+f_2x^2+f_3x^3+\cdots.$$
Thus the elements $g(x)$ are multiplicatively invertible (with inverse $\frac{1}{g(x)}$), and $f(x)$ is compositionally invertible, with inverse $\bar{f}(x)$, where $f(\bar{f}(x))=x$ and $\bar{f}(f(x))=x$. Thus $\bar{f}(x)$ is the solution $u(x)$ of the equation $f(u)=x$ where $u(0)=0$.

The matrix corresponding to the pair $(g(x), f(x))$ is then the invertible lower-triangular matrix with $(n,k)$-th element given by
$$t_{n,k}=[x^n] g(x) f(x)^k,\quad n,k \ge 0,$$ where $[x^n]$ is the functional on $\mathbb{Z}[[x]]$ that extracts the coefficient of $x^n$ in the power series to which it is applied \cite{method}.

The group law is given by
$$(g, f) \cdot (u, v) = (g.u(f), v(f)),$$ and the inverse is given
by
$$(g, f)^{-1}=\left(\frac{1}{g(\bar{f})}, \bar{f}\right).$$

In the matrix representation, the group product becomes ordinary matrix multiplication.

The ``fundamental theorem of Riordan arrays'' is the statement that
$$(g(x), f(x))\cdot h(x)= g(x)h(f(x)).$$
In matrix terms, this corresponds to the fact the $g(x)h(f(x))$ is the generating function of the sequence obtained by multiplying the vector $(h_0,h_1,h_2,\ldots)$ as a column vector by the matrix representing the array $(g(x), f(x))$. Thus each Riordan array $(g(x), f(x))$ with integer coefficients provides a mapping
$$\mathbb{Z}^{\mathbb{N}} \longrightarrow \mathbb{Z}^{\mathbb{N}}.$$
$$  a_n \mapsto b_n=\sum_{k=0}^n [x^n]g(x)f(x)^k a_k.$$
In this way each Riordan array defines a transformation on the space of integer sequences.
\begin{example}
The Riordan array $\left(\frac{1}{1-x},x\right)$ is the partial sum transformation. We have
$$b_n=\sum_{k=0}^n [x^n]\frac{1}{1-x}x^k a_k=\sum_{k=0}^n [x^{n-k}]\frac{1}{1-x}a_k=\sum_{k=0}^n [k \le n] a_k=\sum_{k=0}^n a_k.$$
Here, we have used the Iverson notation $[\mathcal{P}]$ which evaluates to $1$ if the statement $\mathcal{P}$ is true, and $0$ otherwise \cite{Concrete}.

The corresponding mapping of generating functions is given by
$$g(x) \mapsto \frac{1}{1-x} g(x).$$
The array $\left(\frac{1}{1-x}, x\right)$ is a member of the \emph{Appell subgroup} of the Riordan group. This is the subgroup of elements of the form $(g(x), x)$.
\end{example}
\begin{example} The Riordan array $\left(\frac{1}{1-x},  \frac{x}{1-x}\right)$ is the binomial transform.
Thus we have
$$b_n = \sum_{k=0}^n [x^n]\frac{1}{1-x}\left(\frac{x}{1-x}\right)^k a_k = \sum_{k=0}^n \binom{n}{k} a_k.$$
The corresponding mapping of generating functions is given by
$$g(x) \mapsto \frac{1}{1-x}g\left(\frac{x}{1-x}\right).$$
\end{example}
\begin{example} The Riordan array $(1, xc(x))=(1, x(1-x))^{-1}$ is the Catalan matrix \seqnum{A106566}. Here,
$c(x)=\frac{1-\sqrt{1-4x}}{2x}$ is the generating function of the Catalan numbers $C_n=\frac{1}{n+1}\binom{2n}{n}$.  The matrix has general element $\frac{k+0^{n+k}}{n+0^{n*k}}\binom{2n-k-1}{n-k}$. Thus we have
$$b_n = \sum_{k=0}^n [x^n](xc(x))^k a_k = \sum_{k=0}^n [x^{n-k}]c(x)^ka_k = \sum_{k=0}^n \frac{k+0^{n+k}}{n+0^{n*k}}\binom{2n-k-1}{n-k} a_k.$$
In terms of generating functions, we have
$$ g(x) \mapsto g(xc(x)).$$
\end{example}
Note that there are many so-called ``Catalan matrices'', including $(c(x), xc(x))$, $(1, xc(x)^2)$ and $(c(x), xc(x)^2)$. In this note we shall refer to $(1, xc(x))$ exclusively as the Catalan matrix.

Many examples of sequences and  Riordan arrays are documented in the On-Line Encyclopedia of Integer Sequences (OEIS) \cite{SL1, SL2}. Sequences are frequently referred to by their
OEIS number. For instance, the binomial matrix $\mathbf{B}=\left(\frac{1}{1-x}, \frac{x}{1-x}\right)$ (``Pascal's triangle'') is \seqnum{A007318}. The Catalan matrix $(1, xc(x))$ is \seqnum{A106566}. In the sequel we will not distinguish between an array pair $(g(x), f(x))$ and its matrix representation. Riordan arrays are infinite in extent; we display suitable truncations of the examples that we use.

For a sequence $a_n$, we call the sequence of determinants $h_n=|a_{i+j}|_{0 \le i,j \le \infty}$ the Hankel transform of $a_n$ \cite{Hankel}. As an example, the Hankel transform of the Catalan numbers $C_n$, and that of the once-shifted Catalan numbers $C_{n+1}$ are both given by the all ones sequence $1,1,1,\ldots$. We use the notation $H(x)$ for the generating function of the Hankel transform $h_n$. 

In the text below, we  shall use the same notation $\mathbf{C}$ for the transformation that we shall define, whether we apply it to the terms of a sequence or to the generating function of that sequence. The context will make the usage clear. The ``\textbf{C}'' in the notation stands for \emph{central}.

\section{The ``half'' of a matrix}
For a lower triangular matrix $(t_{n,k})_{0 \le n,k \le \infty}$ we define two ``half'' matrices \cite{halves, half}. The \emph{vertical} half of the matrix is the matrix $(t_{2n-k,n})$ and the \emph{horizontal} half is the matrix $(t_{2n,n+k})$. In terms of the Riordan group, we have
\begin{lemma}
Given a Riordan array $M=(G(x), F(x))$, its vertical half $V$ is the Riordan array
$$V=\left(\frac{\phi(x)\phi'(x) G(\phi(x))}{F(\phi(x))}, \phi(x)\right)=\left(\frac{x \phi'(x) G(\phi(x))}{\phi(x)}, \phi(x)\right),$$ where
$$\phi(x)=\text{Rev}\left(\frac{x^2}{F(x)}\right).$$
Given a Riordan array $(G(x), F(x))=(G(x), xh(x))$, its horizontal half $H$ is the Riordan array
$$H=\left(\frac{\phi'}{h(\phi)}, \phi\right)\cdot (G, F)=\left(\frac{x \phi'}{\phi}, \phi\right)\cdot (G, F)$$ where
$$\phi(x)=\text{Rev}\left(\frac{x^2}{F(x)}\right).$$
The matrices $V$ and $H$ are related by
$$V^{-1} \cdot H = (1, F(x)).$$
\end{lemma}

\section{The transformation $\mathbf{C}$}
In order to define the transformation $\mathbf{C}$, we begin with an example of its construction.
\begin{example} In this example, we shall calculate the transform of the sequence $1,1,1,\ldots$ with generating function $g(x)=\frac{1}{1-x}$.
To start off, we form the Riordan array $(g(x), x)$, an element of the Appell subgroup. In this case, this corresponds to the matrix that begins
$$\left(
\begin{array}{cccccccc}
 1 & 0 & 0 & 0 & 0 & 0 & 0 & 0 \\
 1 & 1 & 0 & 0 & 0 & 0 & 0 & 0 \\
 1 & 1 & 1 & 0 & 0 & 0 & 0 & 0 \\
 1 & 1 & 1 & 1 & 0 & 0 & 0 & 0 \\
 1 & 1 & 1 & 1 & 1 & 0 & 0 & 0 \\
 1 & 1 & 1 & 1 & 1 & 1 & 0 & 0 \\
 1 & 1 & 1 & 1 & 1 & 1 & 1 & 0 \\
 1 & 1 & 1 & 1 & 1 & 1 & 1 & 1 \\
\end{array}
\right).$$
We then take the inverse of this matrix, giving $(1-x, x)$, which begins
$$\left(
\begin{array}{cccccccc}
 1 & 0 & 0 & 0 & 0 & 0 & 0 & 0 \\
 -1 & 1 & 0 & 0 & 0 & 0 & 0 & 0 \\
 0 & -1 & 1 & 0 & 0 & 0 & 0 & 0 \\
 0 & 0 & -1 & 1 & 0 & 0 & 0 & 0 \\
 0 & 0 & 0 & -1 & 1 & 0 & 0 & 0 \\
 0 & 0 & 0 & 0 & -1 & 1 & 0 & 0 \\
 0 & 0 & 0 & 0 & 0 & -1 & 1 & 0 \\
 0 & 0 & 0 & 0 & 0 & 0 & -1 & 1 \\
\end{array}
\right).$$

We now take the binomial transform of this matrix. This involves multiplying by the Riordan array $\left(\frac{1}{1-x}, \frac{x}{1-x}\right)$, with representation $\left(\binom{n}{k}\right)$, the binomial matrix (Pascal's triangle).
We have
$$\left(\frac{1}{1-x}, \frac{x}{1-x}\right)\cdot (1-x, x)=\left(\frac{1}{1-x}\left(1-\frac{x}{1-x}\right), \frac{x}{1-x}\right)=\left(\frac{1-2x}{(1-x)^2}, \frac{x}{1-x}\right).$$ The corresponding matrix begins
$$\left(
\begin{array}{cccccccc}
 1 & 0 & 0 & 0 & 0 & 0 & 0 & 0 \\
 0 & 1 & 0 & 0 & 0 & 0 & 0 & 0 \\
 -1 & 1 & 1 & 0 & 0 & 0 & 0 & 0 \\
 -2 & 0 & 2 & 1 & 0 & 0 & 0 & 0 \\
 -3 & -2 & 2 & 3 & 1 & 0 & 0 & 0 \\
 -4 & -5 & 0 & 5 & 4 & 1 & 0 & 0 \\
 -5 & -9 & -5 & 5 & 9 & 5 & 1 & 0 \\
 -6 & -14 & -14 & 0 & 14 & 14 & 6 & 1 \\
\end{array}
\right).$$
The central elements (that is, the $(2n,n)$-th elements) of this matrix are the Catalan numbers $1,1,2,5,14,\ldots$ or $C_n=\frac{1}{n+1}\binom{2n}{n}$ with generating function $c(x)=\frac{1-\sqrt{1-4x}}{2x}$.
We say that the $\mathbf{C}$ transform of the sequence $1,1,1,\ldots$ is this central sequence, that is we have
$$1,1,1,\ldots \xlongrightarrow{\mathbf{C}} 1,1,2,5,14,\ldots,$$ or, in terms of generating functions,
$$\frac{1}{1-x} \xlongrightarrow{\mathbf{C}} c(x)=\frac{1-\sqrt{1-4x}}{2x}.$$

Note that the vertical half of this last matrix is the matrix $(c(x), xc(x))$ that begins
$$V=\left(
\begin{array}{ccccccc}
 1 & 0 & 0 & 0 & 0 & 0 & 0 \\
 1 & 1 & 0 & 0 & 0 & 0 & 0 \\
 2 & 2 & 1 & 0 & 0 & 0 & 0 \\
 5 & 5 & 3 & 1 & 0 & 0 & 0 \\
 14 & 14 & 9 & 4 & 1 & 0 & 0 \\
 42 & 42 & 28 & 14 & 5 & 1 & 0 \\
 132 & 132 & 90 & 48 & 20 & 6 & 1 \\
\end{array}
\right)$$  while the horizontal half is
$$(c(x), x c(x)^2)$$ which begins
$$H=\left(
\begin{array}{ccccccc}
 1 & 0 & 0 & 0 & 0 & 0 & 0 \\
 1 & 1 & 0 & 0 & 0 & 0 & 0 \\
 2 & 3 & 1 & 0 & 0 & 0 & 0 \\
 5 & 9 & 5 & 1 & 0 & 0 & 0 \\
 14 & 28 & 20 & 7 & 1 & 0 & 0 \\
 42 & 90 & 75 & 35 & 9 & 1 & 0 \\
 132 & 297 & 275 & 154 & 54 & 11 & 1 \\
\end{array}
\right).$$
The $\mathbf{C}$ transform of the sequence $1,1,1,\ldots$, or $[x^n]\frac{1}{1-x}$, which is $C_n$, is thus given by the first column of both $H$ and $V$.
\end{example}
A constructive process to arrive at the $\mathbf{C}$ transform of a sequence whose generating function is $g(x)$ is thus as follows.
\begin{itemize}
\item Form the Appell element $(g(x), x)$.
\item Take its inverse to get $\left(\frac{1}{g(x)}, x\right)$.
\item Form the product $\left(\frac{1}{1-x}, \frac{x}{1-x}\right)\cdot \left(\frac{1}{g(x)},x\right)=\left(\frac{1}{1-x}\frac{1}{g\left(\frac{x}{1-x}\right)}, \frac{x}{1-x}\right)$.
\item The $\mathbf{C}$ transform we seek is then the initial column of the half matrices (vertical and horizontal) of $\left(\frac{1}{1-x}\frac{1}{g\left(\frac{x}{1-x}\right)}, \frac{x}{1-x}\right)$.
\end{itemize}
\begin{proposition} The $\mathbf{C}$ transform of $g(x)$ is given by
$$\mathbf{C}(g(x))=\frac{1}{\sqrt{1-4x}g(xc(x)^2)}.$$
\end{proposition}
\begin{proof} We let
$$(G(x), F(x))=\left(\frac{1}{1-x}\frac{1}{g\left(\frac{x}{1-x}\right)}, \frac{x}{1-x}\right).$$
Then we have
$$\phi(x)=\text{Rev}\left(\frac{x^2}{\frac{x}{1-x}}\right)=\text{Rev}(x(1-x))=xc(x).$$
Then $$\frac{x\phi'(x)}{\phi(x)}=\frac{x (x c(x))'}{x c(x)}=\frac{1}{c(x)\sqrt{1-4x}}.$$
We also have
$$G(\phi(x))=\frac{1}{1-xc(x)}\frac{1}{g\left(\frac{xc(x)}{1-xc(x)}\right)}=c(x)\frac{1}{g(x c(x)^2)},$$ where we have used the equality $c(x)=\frac{1}{1-xc(x)}$.
Thus we obtain
$$\frac{x \phi'(x) G(\phi(x))}{\phi(x)}=\frac{1}{c(x)\sqrt{1-4x}}.c(x)\frac{1}{g(x c(x)^2)}=\frac{1}{\sqrt{1-4x}g(xc(x)^2)},$$
as required.
\end{proof}
If $a_n$ is the sequence with generating function $g(x)$, we let $a_n^*$ denote the sequence with generating function $\frac{1}{g}$.
\begin{corollary} The $\mathbf{C}$ transform of the sequence $a_n$ is given by $b_n$ where
$$ b_n = \sum_{k=0}^n \binom{2n}{n-k}a_k^*.$$
\end{corollary}
\begin{proof} Using the fundamental theorem of Riordan arrays, we have
$$\frac{1}{\sqrt{1-4x}g(xc(x)^2)}=\left(\frac{1}{\sqrt{1-4x}}, xc(x)^2\right)\cdot \frac{1}{g(x)}.$$
The Riordan array $\left(\frac{1}{\sqrt{1-4x}}, xc(x)^2\right)$ has general term $\binom{2n}{n-k}$.
\end{proof}
\begin{corollary} If $h(x)=\mathbf{C}(g(x))$ then we have
$$g(x)=\frac{1}{\frac{1-x}{1+x}h\left(\frac{x}{(1+x)^2}\right)}.$$
\end{corollary}
\begin{proof} This follows since we have
$$\left(\frac{1}{\sqrt{1-4x}}, xc(x)^2\right)^{-1}=\left(\frac{1-x}{1+x}, \frac{x}{(1+x)^2}\right).$$
\end{proof}
The array with general term $\binom{2n}{n-k}$ begins
$$\left(
\begin{array}{cccccc}
 1 & 0 & 0 & 0 & 0 & 0 \\
 2 & 1 & 0 & 0 & 0 & 0 \\
 6 & 4 & 1 & 0 & 0 & 0 \\
 20 & 15 & 6 & 1 & 0 & 0 \\
 70 & 56 & 28 & 8 & 1 & 0 \\
 252 & 210 & 120 & 45 & 10 & 1 \\
\end{array}
\right).$$
This is \seqnum{A094527}.
\begin{corollary} With the assumptions of the last corollary, we have that $a_n$ is the reciprocal sequence to the sequence
$$a_n^*=\sum_{k=0}^n (-1)^{n-k} \frac{2n+0^n}{n+k+0^{n+k}}\binom{n+k}{2k}b_k.$$
\end{corollary}
\begin{proof} This follows since the general term of the array $\left(\frac{1-x}{1+x}, \frac{x}{(1+x)^2}\right)$ is given by $$(-1)^{n-k} \frac{2n+0^n}{n+k+0^{n+k}}\binom{n+k}{2k}.$$
\end{proof}
The Riordan array $\left(\frac{1-x}{1+x}, \frac{x}{(1+x)^2}\right)$ begins
$$\left(
\begin{array}{ccccccc}
 1 & 0 & 0 & 0 & 0 & 0 & 0 \\
 -2 & 1 & 0 & 0 & 0 & 0 & 0 \\
 2 & -4 & 1 & 0 & 0 & 0 & 0 \\
 -2 & 9 & -6 & 1 & 0 & 0 & 0 \\
 2 & -16 & 20 & -8 & 1 & 0 & 0 \\
 -2 & 25 & -50 & 35 & -10 & 1 & 0 \\
 2 & -36 & 105 & -112 & 54 & -12 & 1 \\
\end{array}
\right).$$ This is \seqnum{A110162}.
\begin{corollary}
We have $$\mathbf{C}(g(x))=(1, xc(x))\cdot \frac{1}{(1-2x)g\left(\frac{x}{1-x}\right)}.$$
\end{corollary}
\begin{proof}
This follows since
$$\left(\frac{1}{\sqrt{1-4x}}, xc(x)^2\right)=(1, xc(x))\cdot \left(\frac{1}{1-2x},\frac{x}{1-x}\right).$$
\end{proof}
The array $\left(\frac{1}{1-2x},  \frac{x}{1-x}\right)$ is \seqnum{A055248} with general term
$$\sum_{j=0}^{n-k} \binom{k+j-1}{j}2^{n-k-j}= \sum_{j=0}^n \binom{n}{k+j}=\sum_{j=k}^n \binom{n}{j}.$$
The last equality arises since
$$\left(\frac{1}{1-2x},\frac{x}{1-x}\right)=\left(\frac{1}{1-x},\frac{x}{1-x}\right)\cdot \left(\frac{1}{1-x}, x\right).$$
Thus applying the matrix $\left(\frac{1}{1-2x},\frac{x}{1-x}\right)$ to a sequence returns the binomial transform of the partial sums of the original sequence. Thus the $\mathbf{C}$ transform returns the Catalan transform of the binomial transform of the partial sums of the reciprocal of the original sequence. We have
$$\mathbf{C}(g(x))=(1, xc(x))\cdot \left(\frac{1}{1-x},\frac{x}{1-x}\right)\cdot \left( \frac{1}{1-x},x\right)\cdot \frac{1}{g(x)}.$$
We have that $$(1, xc(x))\cdot \left(\frac{1}{1-x},\frac{x}{1-x}\right)=(c(x), xc(x)^2)$$ and thus
$$\mathbf{C}(g(x))=(c(x), xc(x)^2)\cdot \frac{1}{(1-x)g(x)}.$$ The array $(c(x), xc(x)^2)$ is \seqnum{A039599}, with general term $\frac{2k+1}{n+k+1}\binom{2n}{n-k}$.
Thus we have
$$b_n = \sum_{k=0}^n \frac{2k+1}{n+k+1}\binom{2n}{n-k} \sum_{i=0}^k  a_i^*.$$

\section{Some simple transforms}
\begin{example}
We show below some simple transforms. We include in the third column the Hankel transform of the transformed sequence.
\begin{center}
\begin{tabular}{|c|c|c|}
\hline $a_n$ & $b_n$  & $h_n$\\
\hline $(-1)^n$ & $\binom{2n+1}{n+1}$ & $1,1,1,\ldots$ \\
\hline $0^n$ & $\binom{2n}{n}$ & $2^n$  \\
\hline $1,1,1,\ldots$ & $C_n$ & $1,1,1,\ldots$  \\
\hline $2^n$ & $-0^n-\binom{2n-1}{n+1}$ & $2^{n+1}\cos\left(-\frac{\pi (n+1)}{3}\right)$ \\ \hline
\end{tabular}
\end{center}
We show the equivalent information below, this time in terms of generating functions.
\begin{center}
\begin{tabular}{|c|c|c|}
\hline $g(x)$ & $\mathbf{C}(g(x))$  & $H(x)$\\
\hline $\frac{1}{1+x}$ & $\frac{1+xc(x)^2}{\sqrt{1-4x}}$ & $\frac{1}{1-x}$ \\
\hline $1$ & $\frac{1}{\sqrt{1-4x}}$ & $\frac{1}{1-2x}$  \\
\hline $\frac{1}{1-x}$ & $c(x)$ & $\frac{1}{1-x}$  \\
\hline $\frac{1}{1-2x}$ & $\frac{-1+3x+\sqrt{1-4x}}{x\sqrt{1-4x}} $ & $\frac{1-4x}{1-2x+4x^2}$ \\ \hline
\end{tabular}
\end{center}
In the final table of this section, we show some further examples of the transform $\mathbf{C}$ acting on the displayed sequences, along with the Hankel transforms of the image sequences.
\begin{center}
\begin{tabular}{|c|c|c|}
\hline $g(x)$ & $\mathbf{C}(g(x))$  & $H(x)$\\
\hline $\frac{1-x}{1+x}$ & $\frac{1}{1-4x}$  & $1$\\
\hline $\frac{1-2x}{1+x}$ & $\frac{1+3\sqrt{1-4x}}{2\sqrt{1-4x}(2-9x)}$ & $\frac{1}{1+x}$ \\
\hline $\frac{1+x}{1-x}$ & $1$ & $1$  \\
\hline $\frac{1+x}{1-x-x^2}$ & $-\sum_{n=0}^{\infty}\binom{2n-1}{n+1}x^n $ & $\frac{1-3x+2x^2-x^3}{(1-x+x^2)^2}$  \\
\hline $\frac{1}{1-x^2}$ & $\sum_{n=0}^{\infty}C_{n+1}x^n=\frac{c(x)-1}{x}$ & $\frac{1}{1-x}$ \\ \hline
\end{tabular}
\end{center}

\end{example}

\section{The family $\frac{1+ ax}{1+bx}$}
We start by considering the transforms of sequences with generating functions of the form
$$g(x)=\frac{1+ ax}{1+bx}.$$
We find that
$$\mathbf{C}(g(x))=\frac{\sqrt{1 - 4 x}(a - b) + 2x(a - 1)(b - 1) + a + b}{2 \sqrt{1 - 4x}(x(a - 1)^2 + a)}.$$
This may be expressed alternatively as
$$\mathbf{C}(g(x))=\frac{1+(b-1)x c(x)}{((1-2 xc(x))(1+(a-1)xc(x))}.$$
We then have the following conjecture regarding the Hankel transforms of the image sequences.
\begin{conjecture}
The Hankel transforms of the $\mathbf{C}$ transform of sequences with generating functions $\frac{1+ ax}{1+bx}$ have their generating functions given by
$$\frac{1-b(a+b)x}{1-2(1+ab)x+(a+b)^2x^2}.$$
\end{conjecture}
We give several examples of this family.
\begin{example}
When $a=-2, b=1$, the expansion of $\frac{1-2x}{1+x}$ will have as image the sequence \seqnum{A141223} with generating function $\frac{1}{\sqrt{1-4x}(1-3xc(x))}=\frac{1}{(1-2xc(x))(1-3xc(x))}$. The Hankel transform of this image is then $h_n=(-1)^n$ with generating function $\frac{1}{1+x}$.
\end{example}
\begin{example}
When $a=1, b=2$, the expansion of $\frac{1+x}{1+2x}$ will have as image the sequence \seqnum{A029651} with generating function $\frac{1+xc(x)}{1-2xc(x)}$. This sequence gives the central elements of the $(1,2)$-Pascal triangle \seqnum{A029635}. The Hankel transform of this image has generating function $\frac{1-6x}{(1-3x)^2}$.

Note that in the special case where $b=a+1$, the generating function of the Hankel transform will have the form
$$H(x)=\frac{1-(a+1)(2a+1)x}{1-2(a(a+1)+1)x+(2a+1)^2x^2}.$$ For example, when $a=-2, b=-1$ we get 
$$H(x)=\frac{1}{1-3x}.$$ In fact, in this case we find that the $\mathbf{C}$ transform of $[x^n] \frac{1-2x}{1-x}$, or
$$1, -1, -1, -1, -1, -1, -1, -1, -1, -1, -1,\ldots$$ is given by \seqnum{A007854}
$$1, 3, 12, 51, 222, 978, 4338, 19323, 86310, 386250,\ldots$$ with generating function 
$$\mathbf{C}\left(\frac{1-2x}{1-x}\right)=\frac{1}{1-3xc(x)}.$$ The Hankel transform of  \seqnum{A007854} is $3^n$. 

We note that the $(1,2)$-Pascal matrix is the almost Riordan array \cite{almost} of the first order defined by
$$\left(\frac{1}{1-x}; \frac{2-x}{(1-x)^2}, \frac{x}{1-x}\right).$$
\end{example}
\begin{example} When $a=1, b=3$, the expansion of $\frac{1+x}{1+3x}$ will have as image the sequence \seqnum{A100320}, with generating function $\frac{1+2xc(x)}{1-2xc(x)}$. As a moment sequence, this has the integral representation
$$b_n=\frac{2}{\pi} \int_{0}^4 \frac{x^n \sqrt{x(4-x)}}{x(4-x)}\,dx-0^n.$$
The Hankel transform of this image has generating function $\frac{1-12x}{(1-4x)^2}$.

We note that in this case, the image sequence is also the image of $[x^n]\frac{1+2x}{1-2x}=1, 4, 8, 16, \ldots$ or \seqnum{A151821} under the Catalan matrix $(1, xc(x))$, or \seqnum{A106566}.
\end{example}
\begin{example} When $a=1, b=4$,  the expansion of $\frac{1+x}{1+4x}$ will have as image the sequence \seqnum{A029609}, with generating function $\frac{5 \sqrt{1-4x}-3(1-4x)}{2{1-4x}}=\frac{1+3xc(x)}{1-2xc(x)}$. The Hankel transform of this image sequence has generating function $\frac{1-20x}{(1-5x)^2}$. The image sequence is the central sequence of the $(2,3)$-Pascal triangle \seqnum{A029600}, which may be defined as the almost Riordan array of first order defined by
$$\left(\frac{1+x}{1-x}; \frac{3-x}{(1-x)^2},\frac{x}{1-x}\right).$$
\end{example}

\section{The family $\frac{1+ax}{1-bx^2}$}
We next consider the transforms of sequences whose generating functions are of the form
$$\frac{1+ax}{1-bx^2}.$$
The $\mathbf{C}$ transform of such a sequence will have its generating function given by
$$\frac{\sqrt{1-4x}(b+x(a(b+1)-2b))+2x^2(a-1)(b-1)+x(4b-a(b-1))-b}{2 \sqrt{1-4x}(x(a-1)^2+a)}.$$
This may be alternatively expressed as
$$\frac{c(x)(1-2xc(x)+(1-b)x^2c(x)^2)}{(1-2xc(x))(1+(a-1)xc(x))}.$$
We then have the following conjecture.
\begin{conjecture} The Hankel transform of the $\mathbf{C}$ transform of the sequence with generating function $\frac{1+ax}{1-bx^2}$ has its generating function given by
$$\frac{1-3bx+b^2(2+b)x^2-b^4x^3}{1-2(1+b)x+(a^2+4b-2a^2b+b^2+a^2b^2)x^2-2b^2(1+b)x^3+b^4x^4}.$$
\end{conjecture}
\begin{example}  When $a=2, b=0$ we find that the transform of the sequence with generating function $1+2x$ is the sequence \seqnum{A072547} with generating function $\frac{1-x+\sqrt{1-4x}}{(2+x)\sqrt{1-4x}}$ which is also the transform of the Jacobsthal numbers \seqnum{A078008} with generating function $\frac{1-x}{1-x-2x^2}=\frac{1-x}{(1+x)(1-2x)}$ by the Catalan array $(1, xc(x))$. The integral representation of this moment sequence is given by
$$b_n=\frac{1}{\pi}\int_0^4 \frac{x^n (x-1)}{(1+2x) \sqrt{x(4-x)}}\,dx+\left(-\frac{1}{2}\right)^n.$$
Its Hankel transform is the sequence that begins
$$1, 2, 0, -8, -16, 0, 64, 128, 0, -512, -1024,\ldots$$ with generating function $\frac{1}{1-2x+4x^2}$.
\end{example}
\begin{example} When $a=1, b=-1$ we find that the transform of the sequence with generating function $\frac{1+x}{1+x^2}$ is the sequence $0^n+\sum_{k=0}^n \binom{n}{k}\binom{n}{k+1}$ with generating function $\frac{(1-4x-\sqrt{1-4x})(1-2x)}{2x(4x-1)}$. This sequence begins
 $$1, 1, 4, 15, 56, 210, 792,\ldots,$$ with moment representation
 $$b_n = \frac{1}{\pi} \int_0^4 \frac{x^n(x-2)}{2\sqrt{x(4-x)}}\,dx+0^n.$$ This sequence is a variant of \seqnum{A158500}.
 The Hankel transform of the image sequence $b_n$ has generating function $\frac{1+3x+x^2-x^3}{\left(1+x^2\right)^2}$. This sequence begins
$$1, 3, -1, -7, 1, 11, -1, -15, 1, 19, -1,\ldots.$$
The absolute value sequence (beginning with $0$)
$$0, 1, 3, 1, 7, 1, 11, 1, 15, 1, 19, 1,\ldots$$ is recorded in the OEIS as \seqnum{A266724}, where it is described as the number of OFF (white) cells in the $n$-th iteration of the ``Rule 59'' elementary cellular automaton starting with a single ON (black) cell.
\end{example}
\begin{example} When $a=1, b=-2$, we find that the sequence $a_n$ with generating function $\frac{1+x}{1+2x^2}$ begins $$1, 1, -2, -2, 4, 4, -8, -8, 16, 16, -32,\ldots.$$ This is $(-1)^{\binom{n+1}{2}}2^{\lfloor \frac{n}{2} \rfloor}$. The reciprocal sequence $a_n^*$ begins 
$$1, -1, 3, -3, 3, -3, 3, -3, 3, -3, 3,\ldots.$$ The $\mathbf{C}$ transform $b_n$, which begins
$$1, 1, 5, 20, 77, 294, 1122, 4290, 16445,\ldots$$ has generating function $\frac{\sqrt{1-4x}(5x-2)+12x^2-11x+2}{2x(4x-1)}$. It has the moment representation
$$b_n=\frac{1}{\pi} \int_0^4 \frac{x^n(2x-5)}{2 \sqrt{x(4-x)}}\,dx+\frac{3}{2}0^n.$$
The Hankel transform of the image sequence $b_n$ has generating function $$H(x)=\frac{1+6x-16x^2}{\left(1+x+4 x^2\right)^2}.$$ We remark that the $\mathbf{C}$ transform of the reciprocal sequence $1, -1, 3, -3, 3, -3, 3, -3,\ldots$ has a Hankel transform with generating function given by 
$$H(x)=\frac{1}{1+x+4x^2}.$$ 

The sequence $b_n$ is a Narayana transform of the sequence
$$1,1, 4, 7, 10, 13, 16,19,\ldots$$ whose generating function is $\frac{1-x+3 x^2}{(1-x)^2}$.
Thus we have
$$b_n=\sum_{k=0}^n \frac{1}{n-k+1}\binom{n-1}{n-k}\binom{n}{k}(3k+3\cdot 0^k-2)=\frac{3n-1}{n+1}\binom{2n-1}{n-1}+0^n.$$
This is essentially \seqnum{A129869}, which gives the number of positive clusters of type $D_n$ \cite{assoc}. We note that the Hankel transform of $b_{n+1}$ has generating function $\frac{1-4x}{1+x+4x^2}$.
\end{example}

\section{The family $\frac{1+ax}{1-x^2}$}
We recall that the INVERT transform of a sequence with generating function $f(x)$ is the sequence with generating function $\frac{x}{1-xf(x)}$. More generally, we shall say that the INVERT$(\alpha)$ transform of a sequence with generating function $f(x)$ is the sequence with generating $\frac{f(x)}{1+\alpha xf(x)}$. Thus the usual INVERT transform is the INVERT$(-1)$ transform.

We have seen that the $\mathbf{C}$ transform of the sequence $1,1,1,\ldots$ with generating function $\frac{1}{1-x}$ is the Catalan numbers $C_n$ with generating function $c(x)$. Note that $\frac{1}{1-x}=\frac{1+x}{1-x^2}$. We then have the following result.
\begin{proposition} The $\mathbf{C}$ transform of the sequence with generating function $\frac{1+ax}{1-x^2}$ is the INVERT$(a-1)$ transform of the Catalan numbers.
\end{proposition}
\begin{proof}
\begin{align*}
\mathbf{C}\left(\frac{1+ax}{1-x^2}\right)&=(1, xc(x))\cdot \left(\frac{1}{1-x}, \frac{x}{1-x}\right)\cdot \left(\frac{1}{1-x},x\right)\cdot \frac{1-x^2}{1+ax}\\
&=(1, xc(x))\cdot \left(\frac{1}{1-x}, \frac{x}{1-x}\right)\cdot \frac{1+x}{1+ax}\\
&=(1, xc(x))\cdot \frac{1}{(1-x)(1+(a-1)x)}\\
&=\frac{1}{(1-xc(x))(1+(a-1)xc(x))}\\
&=\frac{c(x)}{1+(a-1)xc(x)}\end{align*}
\end{proof}
As an example, the sequence \seqnum{A000034} with generating function $\frac{1+2x}{1-x^2}$ which begins
$$1, 2, 1, 2, 1, 2, 1, 2, 1, \ldots$$ is mapped to the Fine numbers \seqnum{A000957}, with generating function
$$F(x)=\frac{c(x)}{1+xc(x)}.$$
We note that the sequence \seqnum{A010872}$(n+1)$ with generating function $\frac{1+2x}{1-x^3}$ is mapped to the sequence \seqnum{A118973}, which has generating function $(1-x)c(x)\cdot F(x)$.

\section{The family $\frac{1+ax}{1-x^3}$}
We now turn our attention to sequences with generating functions of the form $\frac{1+ax}{1-x^3}$.
\begin{example} We start by looking at the transform of $\frac{1+x}{1-x^3}$ which expands to give
$$1, 1, 0, 1, 1, 0, 1, 1, 0, 1, 1,\ldots.$$
The transform will then have generating function
$$\frac{(1-x)(1-2x-\sqrt{1-4x})}{2x^2}=1+xc(x)^3.$$ This expands to give
$$1, 1, 3, 9, 28, 90, 297, 1001, 3432, 11934,$$ which has general term
 $$b_n=0^n+\frac{3nC_n}{n+2}.$$ This image sequence is essentially \seqnum{A000245}. The Hankel transform of the image has generating function
$$\frac{1+x^2-x^3}{\left(1-x+x^2\right)^2},$$ which expands to give
$$1, 2, 2, -1, -5, -5, 1, 8, 8, -1, -11, \ldots.$$ This is \seqnum{A187307}.
\end{example}
\begin{example} We now look at $\frac{1+2x}{1-x^3}$, which expands to give
$$1, 2, 0, 1, 2, 0, 1, 2, 0, 1, 2,\ldots.$$
We have $$\mathbf{C}\left(\frac{1+2x}{1-x^3}\right)=\frac{(1-x)(1-x-(1+x)\sqrt{1-4x})}{2x^2(x+2)}=\frac{1+xc(x)^3}{1+xc(x)}.$$
This expands to give the sequence \seqnum{A118973} which begins
$$1, 0, 2, 5, 16, 51, 168, 565, 1934, 6716,\ldots.$$
We note that this is the Catalan transform of the sequence \seqnum{A028242} (follow $n+1$ by $n$) which begins
$$1, 0, 2, 1, 3, 2, 4, 3, 5, 4, 6, 5, 7,\ldots.$$
The Hankel transform of the image sequence has generating function
$$\frac{1-x+x^2-x^3}{1-3x+8x^2-3x^3+x^4}.$$
\end{example}
In general, we have the following conjecture.
\begin{conjecture} The Hankel transform of the $\mathbf{C}$ transform of the sequence with generating function $\frac{1+a x}{1-x^3}$ has its generating function given by
$$\frac{1+(a-1)x+x^2-x^3}{1-(a+1)x+((a+1)^2-1)x^2-(a+1)x^3+x^4}.$$
\end{conjecture}
In particular, when $a=0$, the Hankel transform has generating function $\frac{1+x^2}{1-x^3}$, which expands to give
$$1,0,1,1,0,1,1,0,1,1,0,\ldots.$$
In this case, the original sequence $a_n$ is
$$1, 0, 0, 1, 0, 0, 1, 0, 0, 1, 0,\ldots,$$ whose $\mathbf{C}$ transform $b_n$ is the sequence \seqnum{A026012} which begins
$$1, 2, 6, 19, 62, 207, 704, 2431, 8502, 30056, 107236,\ldots.$$ The sequence $b_n$ has generating function
$(1+xc(x)^3)c(x)$.
\begin{proposition} The $\mathbf{C}$ transform of $\frac{1+ax}{1-x^3}$ has generating function
$$\frac{1+xc(x)^3}{1+(a-1)xc(x)}.$$
\end{proposition}
\begin{proof}
\begin{align*}
\mathbf{C}\left(\frac{1+ax}{1-x^3}\right)&=(1, xc(x))\cdot \left(\frac{1}{1-x},\frac{x}{1-x}\right)\cdot \left(\frac{1}{1-x},x\right)\cdot \frac{1-x^3}{1+ax}\\
&=(1, xc(x))\cdot \left(\frac{1}{1-x},\frac{x}{1-x}\right)\cdot \frac{1-x^3}{(1-x)(1+ax)}\\
&=(1, xc(x))\cdot \left(\frac{1}{1-x},\frac{x}{1-x}\right)\cdot \frac{1+x+x^2}{1+ax}\\
&=(1, xc(x))\cdot \frac{1-x+x^2}{(1-x)^2}\cdot \frac{1}{1+(a-1)x}\\
&=\frac{1-xc(x)+x^2c(x)^2}{(1-xc(x))^2}\cdot \frac{1}{1+(a-1)xc(x)}\\
&=c(x)^2(1-xc(x)+x^2c(x)^2)\cdot \frac{1}{1+(a-1)xc(x)}\\
&=c(x)^2\left(\frac{1}{c(x)}+x^2c(x)^2\right)\cdot \frac{1}{1+(a-1)xc(x)}\\
&=c(x)(1+x^2c(x)^3)\cdot \frac{1}{1+(a-1)xc(x)}\\
&=\frac{1+xc(x)^3}{1+(a-1)xc(x)}.\end{align*}
\end{proof}
We now note that
$$1+xc(x)^3=(1-x)c(x)^2.$$
Thus the image of $\frac{1+ax}{1-x^3}$ is also given by $\frac{(1-x)c(x)^2}{1+(a-1)xc(x)}$. This give us the following result.
\begin{proposition} The generating function of the image $\frac{1+ax}{1-x^3}$ is equal to $(1-x)c(x)$ times the generating function of the image of $\frac{1+ax}{1-x^2}$.
\end{proposition}

\section{The family $\frac{1-(r-2)x+x^2}{1-sx-x^2}$}
A straightforward calculation shows that
$$\mathbf{C}\left(\frac{1-(r-2)x+x^2}{1-sx-x^2}\right)=\frac{\sqrt{1-4x}-sx}{(1-rx)\sqrt{1-4x}}.$$
However, we find it instructive to arrive at this result by using the Riordan array interpretation of the $\mathbf{C}$ transformation. The family of sequences whose generating function is given by $\frac{1-(r-2)x+x^2}{1-sx-x^2}$ can be defined by the Riordan array
$$\left(\frac{1-(r-2)x+x^2}{1-x^2}, \frac{x}{1-x^2}\right),$$ since we have
$$\left(\frac{1-(r-2)x+x^2}{1-x^2}, \frac{x}{1-x^2}\right)\cdot \frac{1}{1-sx}=\frac{1-(r-2)x+x^2}{1-sx-x^2}.$$
For the $\mathbf{C}$ transform, we are of course more interested in the reciprocal of this generating function. Again, we have a Riordan array definition, this time using the Riordan array
$$ \left(\frac{1-sx-x^2}{(1+x)^2}, \frac{x}{(1+x)^2}\right),$$ since we have
$$\left(\frac{1-sx-x^2}{(1+x)^2}, \frac{x}{(1+x)^2}\right)\cdot \frac{1}{1-rx}=\frac{1-sx-x^2}{1-(r-2)x+x^2}.$$
Thus the $\mathbf{C}$ transform of $[x^n]\frac{1-(r-2)x+x^2}{1-sx-x^2}$ will have its generating function given by
$$\left(\frac{1}{\sqrt{1-4x}}, xc(x)^2\right)\cdot \left(\frac{1-sx-x^2}{(1+x)^2}, \frac{x}{(1+x)^2}\right)\cdot \frac{1}{1-rx}.$$
Calculating the product on the left gives us
$$\left(\frac{\sqrt{1-4x}-sx}{\sqrt{1-4x}},x\right) \cdot \frac{1}{1-rx}= \frac{\sqrt{1-4x}-sx}{(1-rx)\sqrt{1-4x}},$$ as required.
We then have the following conjecture concerning the Hankel transform of the image sequences.
\begin{conjecture} The Hankel transform of the $\mathbf{C}$ transform of the sequence with generating function
$\frac{1-(r-2)x+x^2}{1-sx-x^2}$ has a Hankel transform whose generating function is given by $$\frac{1-s(r+s-2)x}{1+2s(2-r)x+4s^2x^2}.$$
\end{conjecture}
\begin{example} When $r=2, s=-1$, we obtain the sequence with generating function $\frac{1+x^2}{1+x-x^2}$ which begins
$$1, -1, 3, -4, 7, -11, 18, -29, 47, -76, 123,\ldots.$$ This is a signed variant of the Lucas numbers \seqnum{A000032}. The $\mathbf{C}$ transform $b_n$ has generating function $\frac{\sqrt{1-4x}+x}{(1-2x)\sqrt{1-4x}}$, and begins
$$1, 3, 8, 22, 64, 198, 648, 2220, 7872, 28614,\ldots.$$
The Hankel transform of the sequence $b_n$ is given by
$$1, -1, -4, 4, 16, -16, -64, 64, 256, -256, -1024,\ldots,$$ with generating function
$$\frac{1-x}{1+4x^2}.$$
\end{example}
We can express the $\mathbf{C}$ transform of this section in terms of Riordan arrays, as we have
$$\mathbf{C}\left(\frac{1-(r-2)x+x^2}{1-sx-x^2}\right)=\frac{\sqrt{1-4x}-sx}{(1-rx)\sqrt{1-4x}}=\left(1-\frac{s x}{\sqrt{1-4x}}, x\right)\cdot \frac{1}{1-rx}.$$

\section{The family $\frac{1+x^r}{1-x^r}$}
The family of sequences starting with $[x^n]\frac{1+x}{1-x}$ (that is, $r=1$) begins
$$1,2,2,2,\ldots$$ and then they proceed, as $r \in \mathbb{N}$ is incremented, to successively aerate to give
$$1,0,2,0,2,0,\ldots,$$
$$1,0,0,2,0,0,2,0,0,\ldots,$$
$$1,0,0,0,2,0,0,0,2,0,0,\ldots,$$ and so on.
In order to calculate their $\mathbf{C}$ transforms, we apply the matrix $\left(\binom{2n}{n-k}\right)$ to their reciprocals. These reciprocals begin
$$1,-2,2,-2,\ldots$$ and then they proceed to successively aerate to give
$$1,0,-2,0,2,0,\ldots,$$
$$1,0,0,-2,0,0,2,0,0,\ldots,$$
$$1,0,0,0,-2,0,0,0,2,0,0,\ldots,$$ and so on. We give below a table of these transforms, using their generating functions.

\begin{center}
\begin{tabular}{|c|c|}
\hline $g(x)$ & $\mathbf{C}(g(x))$ \\
\hline $\frac{1+x}{1-x}$ & $1$  \\
\hline $\frac{1+x^2}{1-x^2}$ & $\frac{1}{1-2x}$   \\
\hline $\frac{1+x^3}{1-x^3}$ & $\frac{1}{1-3x}$   \\
\hline $\frac{1+x^4}{1-x^4}$ & $\frac{1-2x}{1-4x+2x^2}$  \\
\hline $\frac{1+x^5}{1-x^5}$ & $\frac{1-3x+x^2}{1-5x+5x^2}$ \\
\hline $\frac{1+x^6}{1-x^6}$ & $\frac{1-4x+3x^2}{1-6x+9x^2-2x^2}$ \\ \hline
\end{tabular}
\end{center}
The explanation of the pattern emerging comes from the fact that Riordan arrays of the form
$$\left(\frac{1+\alpha x + \beta x^2}{1+a x + b x^2},\frac{x}{1+a x + b x^2}\right)$$ are the coefficient arrays of families of (constant coefficient) orthogonal polynomials \cite{classical}. Thus we define four families of orthogonal polynomials, $P_n(x), Q_n(x), R_n(x)$ and $S_n(x)$ with coefficient arrays given by, respectively,
$$\left(\frac{1}{1+x}, \frac{x}{(1+x)^2}\right), \left(\frac{1}{(1+x)^2}, \frac{x}{(1+x)^2}\right),$$
$$\left(\frac{1-x}{(1+x)^2}, \frac{x}{(1+x)^2}\right)\quad\text{and}\quad\left(\frac{1-x}{1+x}, \frac{x}{(1+x)^2}\right).$$
Then, if $n$ is odd, the transform of $\frac{1+x^n}{1-x^n}$ is given by the quotient
$$\frac{P_n\left(\frac{1}{x}\right)}{R_n\left(\frac{1}{x}\right)},$$ and if $n$ is even, the transform is given by $$\frac{\frac{1}{x}Q_n\left(\frac{1}{x}\right)}{S_n\left(\frac{1}{x}\right)}.$$

If we evaluate the list of quotients above for $x=2$, we get the fractions
$$1, - \frac{1}{3}, \frac{1}{5}, -3, -\frac{1}{11},\frac{5}{9}, -\frac{7}{13}, \frac{3}{31}, \frac{17}{5}, - \frac{11}{57}, \frac{32}{67},\ldots.$$
The sequence of denominators is then
$$1, 3, 5, 1, 11, 9, 13, 31, 5, 57, 67,\ldots,$$
or \seqnum{A077021}. The numerators
$$1, -1, 1, -3, -1, 5, -7, 3, 17, -11, 23,$$
are an alternatively signed version of \seqnum{A107920} (the Lucas and Lehmer numbers with parameters $\left(1\pm \sqrt{-7}\right)/2$).
In fact, we have that the pairs
$$ (1, 1), (1, 3), (1, 5), (3, 1), (1, 11), (5, 9), (7, 13), (3, 31),\ldots$$ are the solutions of the diophantine equation
$$2^n=7x^2+y^2.$$
The signed sequence of denominators
$$1, -3, -5, 1, 11, 9, -13, -31, -5, 57, 67, -47, -181,\ldots$$ (essentially \seqnum{A002249}) is
$$2\cdot 2^{\frac{n+1}{2}}\cos((n+1)\arctan(\sqrt{7}))=[x^n]\frac{1-4x}{1-x+2x^2}.$$
The sequence
$$ 1, 1, -1, -3, -1, 5, 7, -3, -17, -11,\ldots$$ has general term
$$ \frac{2}{\sqrt{7}}2^{\frac{n+1}{2}} \sin((n+1) \arctan(\sqrt{7}))=[x^n]\frac{1}{1-x+2x^2}.$$
We are not aware of the generating function of the sequence
$$1, 3, 5, 1, 11, 9, 13, 31, 5, 57, 67,\ldots.$$

In like fashion, the ratios for $x=4$ are associated to the diophantine equation
$$4^{n+1}=15x^2+y^2.$$ The ratios are
$$1, - \frac{1}{7},\frac{3}{11} , - \frac{7}{17},\frac{5}{61} , -\frac{33}{7} , -\frac{13}{251},\frac{119}{223} , - \frac{171}{781}, \frac{305}{1673}, -\frac{989}{1451},\ldots.$$
The sequence of denominators $1, 7, 11, 17, 61, 7, 251, 223, 781, 1673, 1451,\ldots$ is the absolute value of the sequence
$$1, -7, -11, 17, 61, -7, -251, -223, 781, 1673, -1451,\ldots$$  with generating function $\frac{1-8 x}{ 1-x+4x^2}$,  with general term $2^{n+2}\cos((n+1)\arctan(\sqrt{15}))$. This is essentially \seqnum{A272931}.
The numerator sequence
$$1, -1, 3, -7, 5, -33, -13, 119, -171, 305, -989,\ldots,$$ is an alternatively signed version
of \seqnum{A106853}, which begins
$$1, 1, -3, -7, 5, 33, 13, -119, -171, 305, 989,\ldots$$ and which has generating function
$$\frac{1}{1-x+4x^2}.$$
In general, the ratios in $x$ are given by
$$\mathbf{C}\left(\frac{1+x^n}{1-x^n}\right)=\frac{[y^n] \frac{1}{1-y+xy^2}}{[y^n]\frac{1-2xy}{1-y+xy^2}}=\frac{\sum_{j=0}^n \binom{n-j}{j}(-x)^j}{\sum_{j=0}^n \binom{n-j}{j}(-x)^j-2x \sum_{j=0}^{n-1}\binom{n-j-1}{j}(-x)^j}.$$
In the limit, as $n \to \infty$, the expansions of $\frac{1+x^n}{1-x^n}$ converge to $0^n$ (due to the ``infinite'' aeration), and thus the limit generating function is $\frac{1}{\sqrt{1-4x}}$. For instance, the transform of $\frac{1+x^{10}}{1-x^{10}}$ is given by $\frac{1-9x+28x^2-35x^3+15x^4-x^5}{1-11x+44x-77x+55x-11x^5}$. The expansion of this generating function
begins
$$1, 2, 6, 20, 70, 252, 924, 3432, 12870, 48620, 184756,\ldots$$ which agrees with $\binom{2n}{n}$ to the number of terms shown (that is, the first $11$ terms). Thus we obtain a sequence of rational approximations to $\frac{1}{\sqrt{1-4x}}$, that is, we have
$$\lim_{n \to \infty} \mathbf{C}\left(\frac{1+x^n}{1-x^n}\right)=\mathbf{C}(1)=\frac{1}{\sqrt{1-4x}}.$$

\section{Pre-images of Narayana polynomials}
We let $$N_{n,k}=\frac{1}{k+1} \binom{n}{k}\binom{n+1}{k}$$ be the (symmetric) Narayana triangle \seqnum{A001263}. We consider the $\mathbf{C}$ pre-images of the sequences $\sum_{k=0}^n N_{n,k}r^k$,  for $r \in \mathbb{N}$. For $r=0,1,2$ we get the sequence of all $1$'s sequence, the sequence $C_{n+1}$ of once shifted Catalan numbers, and $s_{n+1}$, the once-shifted sequence of small Schroeder numbers, respectively. We have
$$[x^n] \frac{1+x+x^2}{1-x^2} \longrightarrow 1,1,1,1,\ldots$$  and
$$[x^n] \frac{1}{1-x^2} \longrightarrow 1,2,5,14,\ldots.$$
Thus we might hope that the pre-images sought would be rational in nature. However, this is not the case.
\begin{proposition} The Narayana polynomials $\sum_{k=0}^n N_{n,k}r^k$ whose generating functions are $\frac{1-(r+1)x-\sqrt{1-2(r+1)x+(r-1)^2x^2}}{2rx^2}$ have a pre-image whose generating function is given by
$$\frac{1-(r-1)x+x^2+\sqrt{1-2(r-1)x+(r^2-6r+3)x^2-2(r-1)x^3+x^4}}{2(1-x^2)}.$$
\end{proposition}
For the unshifted sequences $C_n$, $s_n$ and so on, the relevant generating function is
$$1+\frac{1-(r+1)x-\sqrt{1-2(r+1)x+(r-1)^2x^2}}{2rx}.$$
The pre-image of this generating function for the $\mathbf{C}$ transform is then given by
$$\frac{2rx(1+x)}{(1-x)(1+(r+1)x+x^2-\sqrt{1-2(r-1)x+(r^2-6r+3)x^2-2(r-1)x^3+x^4})}.$$
For instance, the pre-image of the little Schroeder numbers begins
$$1, 1, 0, -1, -4, -11, -30, -83, -236, -689, -2056,\ldots,$$ with generating function
$$\frac{4x(1+x)}{(1-x)(1+3x+x^2-\sqrt{1-2x-5x^2-2x^3+x^4})}.$$

\section{Sequences from mutation effects in trees}
Enumeration in trees is a rich source of integer sequences. One particular aspect is the study of mutations in ordered trees \cite{mutate}. We here examine the pre-images of some sequences arising in this context. We use the abbreviations $B=\frac{1}{\sqrt{1-4x}}$ and $C=c(x)=\frac{1-\sqrt{1-4x}}{2x}$. By the reciprocal of a sequence, we mean the sequence whose generating function is the reciprocal of the generating function of the original sequence.

\begin{center}
\begin{tabular}{|c|c||c|c|c|}
\hline Number & Comment & Pre-image & Number & Comment\\
\hline A000346 & $2\cdot 4^n-\binom{2n+1}{n+1}$ &  $\frac{1-x}{(1+x)^2}$ & A157142 & Signed odd numbers\\
\hline A001700 & $\binom{2n+1}{n+1}$ &   $\frac{1}{1+x}$ & A033999 & $(-1)^n$ \\
\hline A002057 & $[x^n]c(x)^4$ &  $\frac{1}{(1-x)(1+x)^3}$ & $(-1)^n$ A002620 & Signed quarter squares \\
\hline A007852 & $T_0=\frac{1-\sqrt{5-4C}}{2-C}$ & $\frac{1+xc(-x)}{1+x}$ & A099324 & $1, 0, -1, -3, -8, -22,\ldots$ \\
\hline A007856 & $\frac{B}{C}T_0$ & $1-xc(x)=\frac{1}{c(x)}$ & A115140 & Reciprocal of Catalan \\
\hline A097070 & $\frac{B-1}{2}+xB^3$ & $\frac{1-x-x^2+x^3}{1-x+2x^2}$ & & Reciprocal of A004442 \\
\hline A097613 & $\frac{3n+1}{2n+2}\binom{2n}{2}+\frac{0^n}{2}$ & $\frac{1+x}{1+x+x^2}$ & A057078 & $\overline{1,0,-1}$ \\
\hline A114121 & $\frac{\sqrt{1-4x}+1-2x}{2(1-4x)}$ & $1-x^2$ & & $1,0,-1,0,0,0,\ldots$\\
\hline A243585 & $\sum_{k=0}^n \binom{2n}{n-k}\binom{2k}{k}$ & $\sqrt{1-4x}$ & A002420 & Reciprocal of $\binom{2n}{n}$ \\
\hline A257589 & $(2n+1)^2 C_n$ & $\frac{(1-x)^2}{(1+x)(1+4x-x^2)}$ & & $\frac{(-1)^n}{2}\left(F_{3n+4}+F_{3n+1}-2\right)$ \\
\hline
\end{tabular}
\end{center}
The generating function of \seqnum{A000346} is $\frac{1}{2x}\left(\frac{1}{1-4x}-\frac{1}{\sqrt{1-4x}}\right)$ \cite{boundary}. Its Hankel transform is $(-1)^n (2n+1)$, the signed odd numbers. This also coincides with its pre-image under the $\mathbf{C}$ transform.
We have $[x^n]c(x)^4=\frac{4}{n+4}\binom{2n+3}{n}$. Its Hankel transform is given by $(-1)^n \lfloor \frac{n+3}{2} \rfloor$ with generating function $\frac{1-2x-x^3}{\left(1+x^2\right)^2}$.

The sequence \seqnum{A099324} with generating function $\frac{1+\sqrt{1+4x}}{2(1+x)}=\frac{1+xc(-x)}{1+x}$ begins
$$1, 0, -1, -3, -8, -22, -64, -196, -625, -2055, -6917,\ldots.$$ These are the alternating sums of the Catalan variant sequence that begins
$$1, 1, -1, 2, -5, 14, -42, 132, -429, 1430, -4862,\ldots.$$
The Hankel transform of \seqnum{A099324} is the sequence $\overline{1,-1,0}$, with generating function $\frac{1}{1+x+x^2}$.

The Hankel transform of \seqnum{A115140} with generating function $1-xc(x)$ is $(-1)^n (n+1)$ with generating function $\frac{1}{(1+x)^2}$.

The sequence \seqnum{A097070} gives the number of compositions of $n$ into $n$ parts, allowing zeros. It is given by $\frac{1}{2}\left(0^n+(n+1)\binom{2n}{n}\right)$. Its Hankel transform begins
$$1, 5, -14, -26, 43, 63, -88, -116, 149, 185, -226,\ldots,$$ with generating function
$$H(x)=\frac{1+5x-11x^2-11x^3+4x^4}{(1+x^2)^3}.$$  The sequence \seqnum{A097070} is the binomial transform of the sequence \seqnum{A113682}, which has general element
$$\sum_{k=0}^n \binom{n+1}{k+1}\binom{k}{n-k}.$$ The sequence \seqnum{A004442} begins
$$1, 0, 3, 2, 5, 4, 7, 6, 9, 8, 11, 10, 13, 12, 15, 14, 17, 16,\ldots$$ which is a ``pair-reversed'' arrangement of the natural numbers: $n+(-1)^n$. Its generating function is $\frac{1-x+2x^2}{(1-x)(1-x^2)}$.

The $\mathbf{C}$ transform of the sequence $n+(-1)^n$ is of interest in itself. We find that
$$\mathbf{C}\left(\frac{1-x+2x^2}{(1-x)(1-x^2)}\right)=\frac{\sqrt{1-4x}(1-2x-(1-4x)\sqrt{1-4x})}{2x(2-11x+16x^2)},$$ which expands to give the sequence that begins
$$1, 2, 3, 0, -26, -150, -641, -2408, -8402, -27948, -90034,\ldots.$$
The generating function of this sequence can be represented by the continued fraction expression
$$\cfrac{1}{1-2x+
\cfrac{x^2}{1-4x+
\cfrac{x^2}{1+
\cfrac{x^2}{1-4x+
\cfrac{x^2}{1+\cdots}}}}},$$
where the coefficients of $x^2$ are all $1$ and the coefficients of $x$ follow the pattern
$$2,4,0,4,0,4,0,4,0,4,0,\ldots.$$
The Hankel transform of this image sequence is then $(-1)^{\binom{n+1}{2}}$.

An interesting point to note is that the sequence whose generating function can be expressed as
$$\cfrac{1}{1-x+
\cfrac{x^2}{1-4x+
\cfrac{x^2}{1+
\cfrac{x^2}{1-4x+
\cfrac{x^2}{1+\cdots}}}}},$$ or equivalently
$$\frac{\sqrt{1-4x}(1-2x-(1-4x)^{3/2})}{x(1-2x)(3-8x+\sqrt{1-4x})},$$ and which begins
$$1, 1, 0, -5, -24, -90, -312, -1053, -3536, -11934,\ldots$$ is \seqnum{A158499}$(n+1)$, and its $\mathbf{C}$ pre-image begins
$$1, 1, 3, 3, 5, 5, 7, 7, 9, 9, 11,\ldots,$$ with generating function
$$\frac{1+x^2}{(1-x)^2 (1+x)}.$$ The sequence \seqnum{A158499} itself begins
$$1,1, 1, 0, -5, -24, -90, -312, -1053, -3536, -11934,\ldots$$ and has $\mathbf{C}$ pre-image
$$1, 1, 2, 2, 2, 2, 2, 2, 2, 2, 2,\ldots,$$  with generating function $\frac{1+x^2}{1-x}$.
We deduce from this that \seqnum{A158499} can be expressed as
$$b_n = \sum_{k=0}^n \binom{2n}{n-k}(-1)^{\binom{k+1}{2}},$$ since the generating function
$\frac{1-x}{1+x^2}$ expands to give
$$1,-1,-1,1,1,-1,-1,1,1,-1,\ldots$$ whose general term is $(-1)^{\binom{n+1}{2}}$.

\section{Equal Hankel transforms}
We have seen in some previous examples an apparent relationship between the Hankel transform of the $\mathbf{C}$ transform and that of the $\mathbf{C}$ transform of its reciprocal. We now provide an example where the two Hankel transforms are equal. 
\begin{example} We consider the sequence $a_n$ with generating function $\left(\frac{1+x}{1-x}\right)^2$. This sequence \seqnum{A008574} begins 
$$ 1, 4, 8, 12, 16, 20, 24, 28, 32, 36, 40, 44, \ldots.$$
Its reciprocal sequence $a_n^*$ will then have generating function $\left(\frac{1-x}{1+x}\right)^2$, and it begins 
$$1, -4, 8, -12, 16, -20, 24, -28, 32, -36, 40, -44, \ldots.$$
We have $a_n^*=(-1)a_n$. 
We obtain
$$\mathbf{C}\left(\left(\frac{1+x}{1-x}\right)^2\right)=\sqrt{1-4x},$$ and 
$$\mathbf{C}\left(\left(\frac{1-x}{1+x}\right)^2\right)=\frac{1}{(1-4x)^{3/2}}.$$ 
\begin{proposition} The Hankel transforms of the $\mathbf{C}$ transforms of $a_n$ and $a_n^*$ are the same, both equal to the sequence $(2n+1)(-2)^n$ with 
$$H(x)=H^*(x)=\frac{1-2x}{(1+2x)^2}.$$
\end{proposition}
\begin{proof} The $4^{\text{th}}$ inverse binomial transform of the generating function $\frac{1}{(1-4x)^{3/2}}$ is equal to $\sqrt{1+4x}$. This expands to an alternating sign version of the expansion of $\sqrt{1-4x}$. The Hankel transform is unchanged by these two transformations. Thus the Hankel transforms are equal.
\end{proof}
\end{example}
It is interesting to examine the Hankel transforms of the $\mathbf{C}$ transforms of the doubled sequences of $a_n$ and $a_n^*$. Thus we look at the generating functions
$$\left(\frac{1+x^2}{1-x^2}\right)^2\quad \text{and}\quad \left(\frac{1-x^2}{1+x^2}\right)^2$$ of the sequences that we shall denote by $A_n$ and $A_n^*$ respectively. 

The $\mathbf{C}$ transform of $A_n$ has transform 
$$\mathbf{C}\left(\left(\frac{1+x^2}{1-x^2}\right)^2\right)=\frac{1-4x}{(1-2x)^2},$$ and the $\mathbf{C}$ transform of $A_n^*$ has transform 
$$\mathbf{C}\left(\left(\frac{1+x^2}{1-x^2}\right)^2\right)=\frac{(1-2x)^2}{(1-4x)^{3/2}}.$$ 
We obtain, respectively, the Hankel transforms 
$$H(x)=\frac{1+4x^2}{(1-2x)(1+2x)^2},$$ which expands to give 
$$1, -2, 12, -24, 80, -160, 448, -896, 2304, -4608, 11264,\ldots,$$ or 
$$h_n=(-2)^n (2 \lfloor \frac{n}{2} \rfloor +1),$$  and 
$$H^*(x)=\frac{1+8x-36x^2+192x^3-144x^4+128x^5+64x^6}{(1-2x)^3(1+2x)^4}.$$ This expands to give $h_n^*$ which begins 
$$1, 6, -36, 360, -1200, 5600, -15680, 56448, -145152, 456192,\ldots.$$ 
We note that the sequence of ratios $\frac{h_n^*}{h_n}$ then begins
$$1, -3, -3, -15, -15, -35, -35, -63, -63, -99, -99,\ldots.$$ This latter sequence has generating function
$$\frac{1-4x-2x^2-4x^3+x^4}{1-x-2x^2+2x^3+x^4-x^5}=\frac{1-4x-2x^2-4x^3+x^4}{(1+x)^2(1-x)^3}.$$

\section{Conclusion}
Interesting transformations of sequences abound. The significance of the $\mathbf{C}$ transform arises from a number of points:
\begin{itemize}
\item It is easily defined in terms of the notion of the ``half'' of a Riordan array
\item As shown in the note, many interesting sequences have simple $\mathbf{C}$ pre-images
\item Hankel transforms seem to behave nicely in the sense that many images of simple rational generating functions have images whose Hankel transforms are also relatively simple rational generating functions
\item Study of this transform can lead to interesting sequences, as detailed in the examples above
\item Study of this transform is facilitated by the use of Riordan arrays, and in doing so, interesting Riordan arrays emerge.
\end{itemize}
It would appear then that this transform merits further consideration.

\bigskip
\hrule
\noindent 2020 {\it Mathematics Subject Classification}: Primary
15B36; Secondary 11B83, 11C20, 11Y55.
\noindent \emph{Keywords:} Riordan group, Riordan array, central coefficients, Hankel transform, generating function.

\bigskip
\hrule
\bigskip
\noindent (Concerned with sequences
\seqnum{A000032},
\seqnum{A000034},
\seqnum{A000245},
\seqnum{A000346},
\seqnum{A000957},
\seqnum{A001263},
\seqnum{A002249},
\seqnum{A004442},
\seqnum{A007318},
\seqnum{A007854},
\seqnum{A010872},
\seqnum{A026012},
\seqnum{A028242},
\seqnum{A029635},
\seqnum{A029651},
\seqnum{A039599},
\seqnum{A055248},
\seqnum{A072547},
\seqnum{A077021},
\seqnum{A078008},
\seqnum{A094527},
\seqnum{A097070},
\seqnum{A099324},
\seqnum{A100320},
\seqnum{A106566},
\seqnum{A106853},
\seqnum{A107920},
\seqnum{A110162},
\seqnum{A115140},
\seqnum{A118973},
\seqnum{A118973},
\seqnum{A129869},
\seqnum{A141223},
\seqnum{A151821},
\seqnum{A158499},
\seqnum{A158500},
\seqnum{A187307},
\seqnum{A266724}, and
\seqnum{A272931}.)

\end{document}